\documentclass{article}
\usepackage{amssymb,amsthm,amsmath,enumerate}
\usepackage{graphics}

\theoremstyle{plain}

\newtheorem{theorem}{Theorem}

\newcommand*{\C}{\ensuremath{\mathbb{C}}}
\newcommand*{\N}{\ensuremath{\mathbb{N}}}

\begin{document}

\date{}

\author{{\.Zywilla Fechner and L\'aszl\'o Sz\'ekelyhidi}\\ 
{\small\it Institute of Mathematical Finance, Ulm University}, \\{\small\it Helmholtzstrasse 18, 89081 Ulm, Germany,}\\
   {\small\rm e-mail: \tt zfechner@gmail.com}\\
   {\small\it Institute of Mathematics, University of Debrecen,}\\
   {\small\rm e-mail: \tt lszekelyhidi@gmail.com} }

\title{Sine and cosine equations on commutative hypergroups
   \footnotetext{The research was partly supported by the
   Hungarian National Foundation for Scientific Research (OTKA),
   Grant No. K111651.}\footnotetext{AMS (2000)
   Subject Classification: 20N20, 43A62, 39B99}}

\maketitle

\begin{abstract} 
In this paper we describe the solutions of the functional equations expressing the addition theorems for sine and cosine on commutative hypergroups. 
\end{abstract}

\vskip1cm

\section{Introduction}

\hskip.5cm In this paper $\C$ denotes  the set of complex numbers. By a {\it hypergroup} we mean a locally compact hypergroup. The identity element of the hypergroup $K$ will be denoted by $o$.
\vskip.3cm

For basics about hypergroups see the monograph \cite{BlH95}. The detailed study of functional equations on hypergroups started with the papers \cite{MR2042564, MR2107959, MR2161803}. A comprehensive monograph on the subject is \cite{Sze12}. Further results and references on this topic can be found in \cite{MR2209674, MR2282869, MR2272893, MR2309586, MR2431934, MR2805073, SzeVaj12c}. Concerning other similar trigonometric-type functional equations the reader should consult with \cite{MR0219936, MR0239308, MR0291872, MR1741472, MR3184407}.
\vskip.3cm

In this paper we study the {\it sine-cosine functional equation}
\begin{equation}\label{sine}
f(x*y)=f(x) g(y)+f(y) g(x)
\end{equation}
and the {\it cosine-sine functional equation}
\begin{equation}\label{cosine}
g(x*y)=g(x) g(y)-f(x) f(y)
\end{equation}
on an arbitrary commutative hypergroup $K$. In case of both equations we shall always assume that $f,g:K\to\C$ are non-identically zero continuous functions. 
\vskip.3cm

We note that these functional equations are fundamental in the theory of functional equations. In particular, if in \eqref{sine} we have $g=1$, then $f$ is an {\it additive function}, that is
\begin{equation}\label{add}
f(x*y)=f(x) +f(y),
\end{equation}
and if in \eqref{cosine} we have $h=0$, then $g$ is an {\it exponential}, that is
\begin{equation}\label{exp}
g(x*y)=g(x) g(y).
\end{equation}

We shall see that in the case of the sine-cosine equation \eqref{sine} the situation is very sophisticated if $g=m$ is an exponential. In the group case exponentials are never zero, hence we can divide by $m$ and we deduce immediately that $f$ has the form $f=a\cdot m$, where $a$ is additive. Obviously, this is a solution of \eqref{sine} on any commutative hypergroup, but there, in general, we cannot divide by $m$ as exponentials on hypergroups may take zero values. It turns out that on a general commutative hypergroup the solutions $f$ of \eqref{sine} with an exponential $g=m$ produce a new basic function class, which cannot be described directly using exponentials and additive functions. This is a new feature provided by the delicate structure of hypergroups and it seems reasonable to introduce the following definition:
if $K$ is a commutative hypergroup and $m$ is an exponential on $K$, then the function $f:K\to\C$ will be called an {\it $m$-sine function}, if it satisfies
$$
f(x*y)=f(x) m(y)+f(y) m(x)
$$ 
for each $x,y$ in $K$. We call $f$ a {\it sine function}, if it is an $m$-sine function for some exponential $m$. Obviously, we have $f(o)=0$ for every sine function $f$. Additive functions on $K$ are exactly the $1$-sine functions. If $K=G$ is an Abelian group, then for a given exponential $m$ the $m$-sine functions are exactly the functions of the form $f=a\cdot m$, where $a$ is additive. But this is not the case on hypergroups. For instance, let $K$ be the polynomial hypergroup generated by the sequence of polynomials $(P_n)_{n\in \N}$ (see \cite{BlH95}). It is known that all exponential functions on $K$ have the form $n\mapsto P_n(\lambda)$ with some complex number $\lambda$, and all additive functions have the form $n\mapsto c P_n'(0)$ with some complex number $c$ (see \cite{Sze12}). Further, if we define
$$
f(n)=P_n'(\lambda),\hskip.5cm m(n)=P_n(\lambda)
$$
for each $n$ in $\N$ with some complex number $\lambda$, then
it is easy to check that the following equation holds for each $m,n$ in $\N$:
$$
f(n*k)=f(n) m(k)+f(k) m(n),
$$
that is, $f$ is an $m$-sine function. On the other hand, it is easy to see that it does not have the form $n\mapsto c\, P_n'(0)\, P_n(\lambda)$ for any complex $c$.  
\vskip.3cm

In the forthcoming paragraphs we shall describe the solutions of the sine-cosine and the cosine-sine functional equations \eqref{sine} and \eqref{cosine} on arbitrary commutative hypergroups in terms of exponentials and sine functions.

\section{Sine-cosine functional equations on \\hypergroups}

\hskip.5cm In this section we describe the nonzero solutions of the sine-cosine functional equation \eqref{sine} on commutative hypergroups.

\begin{theorem}\label{1}
Let $K$ be a commutative hypergroup, and let $f,g:K\to\C$ be non-identically zero continuous functions satisfying \eqref{sine} for each $x,y$ in $K$. Then there exists a complex number $c\ne 0$ and there are continuous exponentials $M,N:K\to\C$ such that we have one the following possibilities:
\begin{enumerate}[i)]
\item 
$g(x)=M(x)$, and $f$ is an $M$-sine function.
\item 
\begin{equation}\label{sine1}
f(x)=\frac{1}{2c}\, M(x),\hskip1cm g(x)=\frac{1}{2}\, M(x)
\end{equation}
for each $x$ in $K$.
\item 
\begin{equation}\label{sine2}
f(x)=\frac{1}{2c}\, [M(x)-N(x)],\hskip1cm g(x)=\frac{1}{2}\, [M(x)+N(x)]
\end{equation}
for each $x$ in $K$. 
 \end{enumerate}
Conversely, the functions $f,g$ given above are continuous solutions of \eqref{sine} for every nonzero complex number $c$ and continuous complex exponentials $M,N$.
\end{theorem}

\begin{proof}
As the case $i)$ obviously describes a possible solution, hence we will suppose that $g$ is not an exponential.
\vskip.3cm

Suppose first that $g(o)\ne 1$. By substitution $y=o$ into \eqref{1} we get
$$
f(x) (1-g(o))=f(o) g(x),
$$
that is $f(x)= \frac{1}{c} g(x)$ with some complex number $c\ne 0$. It follows from \eqref{1}
$$
2 g(x*y)=2 g(x) 2 g(y)
$$
hence $g=\frac{1}{2} m$ and $f=\frac{1}{2c} m$, where $m$ is an exponential, which is given in $ii)$ with $M=m$.
\vskip.3cm

Now we assume $g(o)=1$. By substitution $y=o$ into \eqref{1} we get
$$
f(x)=f(x) g(o)+f(o) g(x)=f(x)+f(o) g(x)
$$
which implies $f(o) g(x)=0$, hence $f(o)=0$.
\vskip.3cm

We introduce the Cauchy difference: for each $x,y$ in $K$ we define
$$
F(x,y)=f(x*y)-f(x)-f(y)
$$
which can be written as
$$
F(x,y)=f(x) [g(y)-1]+f(y)[g(x)-1].
$$
Obviously, $F$ satisfies
$$
F(x,y)+F(x*y,z)=F(x,y*z)+F(y,z)
$$
for each $x,y,z$ in $K$, by the associativity of the hypergroup operation. After substitution and simplification we get the equation
$$
f(z)[g(x*y)-g(x) g(y)]=f(x)[g(y*z)-g(y) g(z)].
$$
for each $x,y,z$ in $K$. As $f\ne 0$, we have 
$$
g(x*y)=g(x) g(y)+f(x) \varphi(y)
$$
with some continuous $\varphi:K\to\C$. By commutativity, we obtain $\varphi=\lambda f$ with some complex number $\lambda$. We write $\lambda=-d^2$ and we infer
$$
g(x*y)=g(x) g(y)-d f(x) d f(y),
$$
or, with the notation $h= d f$ we obtain \eqref{cosine} for $g$ and $h$. Here $d\ne 0$, as otherwise $g$ is an exponential and we have $i)$.
\vskip.3cm

Then we multiply \eqref{sine} by $d$ and we have the system for the pair $g,h$:
\begin{equation}\label{sinecosine}
h(x*y)=h(x) g(y)+h(y) g(x)
\end{equation}
$$
g(x*y)=g(x) g(y)- h(x) h(y)
$$
for each $x,y$ in $K$. Let for each $x$ in $K$:
$$
M(x)=g(x)+i h(x),\hskip.5cm\text{and}\hskip.5cm N(x)=g(x)-i h(x).
$$
We have for each $x,y$ in $K$
$$
M(x*y)=g(x*y)+i h(x*y)=g(x) g(y)-h(x) h(y)+ i g(x) h(y)+i g(y) h(x)=
$$
$$
(g(x)+i h(x)) (g(y)+i h(y))=M(x) M(y),
$$
and 
$$
N(x*y)=g(x*y)-i h(x*y)= g(x) g(y) -h(x) h(y)- i g(x) h(y)- i g(y) h(x)=
$$
$$
(g(x)- i h(x)) (g(y)- i h(y))=N(x) N(y).
$$
This means that $M,N:K\to\C$ are exponentials. On the other hand, we have
$$
g=\frac{1}{2} (M+N),\hskip.5cm h=\frac{1}{2i} (M-N).
$$
It follows $f=\frac{1}{2di} (M-N)$, and we have $iii)$ with $c=di$.
\vskip.3cm

The converse statement can be verified easily by direct computation.
\end{proof}

\section{Cosine-sine functional equations on \\hypergroups}
\hskip.5cm
In this section we describe the nonzero solutions of the cosine-sine functional equation \eqref{cosine} on commutative hypergroups.

\begin{theorem}\label{2}
Let $K$ be a commutative hypergroup, and let $f,g:K\to\C$ be non-identically zero continuous functions satisfying \eqref{cosine} for each $x,y$ in $K$. Then there exist complex numbers $c\ne0, 1$ and $d\ne 0$, and there are continuous exponentials $M,N:K\to\C$ such that we have one of the following possibilities:
\begin{enumerate}[i)]
\item 
\begin{equation}\label{cosine1}
f(x)=\frac{c}{1-c^2}\, M(x),\hskip1cm g(x)=\frac{1}{1-c^2}\, M(x)
\end{equation}
for each $x$ in $K$.
\item 
\begin{equation}\label{cosine2}
f(x)=\frac{1}{2c}\, M(x),\hskip1cm g(x)=\frac{1}{2}\, M(x)
\end{equation}
for each $x$ in $K$. 
\item
$f$ is an $M$-sine function, and $g(x)=M(x)\pm f(x)$ for each $x$ in $K$.
\item 
\begin{equation}\label{cosine4}
f(x)=\pm\frac{1}{2d i} [M(x)-N(x)],\hskip.2cm g(x)=\pm \frac{\pm d i-\lambda}{2d i}M(x)\pm \frac{\pm d i+\lambda}{2d i}N(x)
\end{equation}
for each $x$ in $K$, where $d^2=1-\lambda^2$, and we choose $+$ or $-$ at each place in the same way. 
\end{enumerate}

Conversely, the functions $f,g$ given above are continuous solutions of equation \eqref{cosine} for any nonzero complex numbers $c,d$, $c\ne \pm1$ and continuous complex exponentials $M,N$.
\end{theorem}

\begin{proof}
First we note that here $g$ is not an exponential, otherwise $f$ is identically zero.
Substituting $y=o$ we have $g(x)(1-g(o))=-f(x)f(o)$, hence if $g(o)\ne 1$, then $g(x)=\frac{1}{c} f(x)$ with some complex number $c\ne 0$. We also have $c\ne \pm 1$ otherwise $g=\pm f$ and substitution into \eqref{cosine} gives that $f=g=0$. It follows from \eqref{cosine}
$$
\frac{1-c^2}{c} f(x*y)=\frac{1-c^2}{c} f(x) \frac{1-c^2}{c} f(y)
$$
implying $f(x)=\frac{c}{1-c^2}\, m(x)$ and $g(x)=\frac{1}{1-c^2}\, m(x)$ with some exponential $m$, which is $i)$ with $M=m$.
\vskip.3cm

Now we assume $g(o)=1$, and in this case \eqref{cosine} implies $f(o)=0$. We define a modified Cauchy difference $G(x,y)=g(x*y)-g(x) g(y)$ such that we have
$$
g(z) G(x,y)+G(x*y,z)=G(x,y*z)+g(x) G(y,z)
$$
for each $x,y,z$ in $K$.
Then, by \eqref{cosine}, it follows
$$
f(x)[f(y*z)-f(y)g(z)]=f(z)[f(x*y)-g(x)f(y)].
$$
As $f\ne 0$ this implies
$$
f(x*y)=f(x) \varphi(y)+f(y) g(x)
$$
with some continuous $\varphi:K\to\C$. Interchanging $x$ and $y$ we have the relation
$$
\varphi(x)=g(x)+2\lambda  f(x)
$$
with some complex number $\lambda$. If $\lambda=0$, then we have $\varphi=g$, and the pair $f,g$ satisfies the sine equation \eqref{sine}. Case $i)$ in Theorem \ref{1} cannot occur. Case $ii)$ in Theorem \ref{1} gives $c=\pm i$, which is included in $i)$ above with $c=\pm i$. Finally, case $iii)$ in Theorem \ref{1} gives $c=\pm i$ which is included in case $iii)$ above with $c=\pm i$.
\vskip.3cm

Now we assume $\lambda\ne 0$, then we have
\begin{equation}\label{tricky}
f(x*y)=f(x)g(y)+2\lambda f(x) f(y)+f(y) g(x).
\end{equation}
We introduce the function
$$
h(x)=g(x)+\lambda f(x),
$$
then a simple calculation shows that
\begin{equation}\label{fsin}
f(x*y)=f(x) h(y)+f(y) h(x)
\end{equation}
and
\begin{equation}\label{hexp}
h(x*y)=h(x) h(y)-(1-\lambda^2) f(x) f(y).
\end{equation}

Equation \eqref{fsin} shows that $f$ and $h$ satisfy the sine-cosine equation \eqref{sine}, hence we have the description of the solutions, we just have to extract the solutions of \eqref{cosine}. But we also have to consider equation \eqref{hexp} which depends on $\lambda$. If $\lambda^2=1$, then $h=m$ is an exponential and $f$ is an $m$-sine function. In this case we have $g=m\pm f$ and substitution into \eqref{cosine} gives that this is a solution indeed, which is covered by case $iii)$ above.
\vskip.3cm

Finally we suppose $\lambda^2\ne 1$. We take $d\ne 0$ with $d^2=1-\lambda^2$, then we have by \eqref{hexp}
$$
h(x*y)=h(x) h(y) - d f(x) d f(y)
$$
and, multiplying \eqref{fsin} by $d$ gives
$$
d f(x*y)=d f(x) h(y) +d f(y) h(x)
$$
for each $x,y$ in $K$. This means that the pair $d f,h$ satisfies the sine-cosine and the cosine-sine functional equations simultaneously, and in this case $h$ is not an exponential, hence we have to consider cases $ii)$ and $iii)$ only, in Theorem \ref{1}. In case $ii)$ we get $c=\pm i$ and, by the definition of $h$
$$
f(x)=\pm \frac{1}{2di}M(x),\hskip1cm g(x)+\lambda f(x)=\frac{1}{2} M(x)
$$
which implies that $f$ and $g$ are constant multiples of each other, hence we have case $i)$ above. In case $iii)$ of Theorem \ref{1} we obtain
$$
f(x)=\frac{1}{2cd} [M(x)-N(x)],\hskip1cm g(x)=\frac{cd-\lambda}{2cd}M(x)+ \frac{cd+\lambda}{2cd}N(x)\,.
$$
Substitution into \eqref{cosine} gives $c=\pm i$ and we have
$$
f(x)=\pm\frac{1}{2d i} [M(x)-N(x)],\hskip1cm g(x)=\pm \frac{\pm d i-\lambda}{2d i}M(x)\pm \frac{\pm d i+\lambda}{2d i}N(x)\,,
$$
where $d^2=1-\lambda^2$, and we choose $+$ or $-$ at each place in the same way, as it is given in case $iv)$ above.
\vskip.3cm

The converse statement can be verified easily by direct computation.
\end{proof}

Using results concerning the form of exponentials on some particular hypergroups discussed in \cite{Sze12} one can obtain explicit forms of sine  functions on certain hypergroups.

\end{document}